\documentclass[11pt]{article}\UseRawInputEncoding
\usepackage{amssymb,amsfonts,amsmath,latexsym,epsf,tikz,url}
\usepackage[usenames,dvipsnames]{pstricks}
\usepackage{pstricks-add}
\usepackage{epsfig}
\usepackage{pst-grad} 
\usepackage{pst-plot} 
\usepackage[space]{grffile} 
\usepackage{etoolbox} 
\usepackage{xcolor}
\makeatletter 
\patchcmd\Gread@eps{\@inputcheck#1 }{\@inputcheck"#1"\relax}{}{}
\makeatother

\newtheorem{theorem}{Theorem}[section]
\newtheorem{proposition}[theorem]{Proposition}
\newtheorem{conjecture}[theorem]{Conjecture}
\newtheorem{corollary}[theorem]{Corollary}
\newtheorem{lemma}[theorem]{Lemma}

\newcommand{\proof}{\noindent{\bf Proof.\ }}
\newcommand{\qed}{\hfill $\square$\medskip}

\begin{document}

\title{ Elliptic Sombor energy of a graph}

\author{
Saeid Alikhani$^{a,}$\footnote{Corresponding author} \and Nima Ghanbari$^a$  \and Mohammad Ali Dehghanizadeh$^b$
}

\date{\today}

\maketitle

\begin{center}

$^a$Department of Mathematical Sciences, Yazd University, 89195-741, Yazd, Iran\\
	
$^b$Department of Mathematics, Technical and Vocational University, Tehran, Iran

\medskip
\medskip
{\tt  ~~
alikhani@yazd.ac.ir, n.ghanbari.math@gmail.com,   Mdehghanizadeh@tvu.ac.ir
 }
\end{center}


\begin{abstract}
Let $G$ be a simple graph with vertex set $V(G) = \{v_1, v_2,\ldots, v_n\}$. The elliptic Sombor matrix of $G$, denoted by $A_{ESO}(G)$, is defined as the $n\times n$ matrix whose $(i,j)$-entry is $(d_i+d_j)\sqrt{d_i^2+d_j^2}$ if $v_i$ and
$v_j$ are adjacent and $0$ for another cases.
Let the eigenvalues of the elliptic Sombor matrix $A_{ESO}(G)$ be  $\rho_1\geq \rho_2\geq \ldots\geq \rho_n$ which are the roots of the elliptic Sombor characteristic polynomial $\prod_{i=1}^n (\rho-\rho_i)$. The elliptic Sombor energy ${E_{ESO}}$ of $G$ is the sum of
absolute values of the eigenvalues of $A_{ESO}(G)$. In this paper, we compute the elliptic Sombor characteristic polynomial and the elliptic Sombor energy for  some  graph classes. We compute the elliptic Sombor energy of cubic graphs of order $10$ and as a consequence, we see that two $k$-regular graphs of the same order may have different elliptic Sombor energy.
\end{abstract}

\noindent{\bf Keywords:} Elliptic Sombor Matrix, Elliptic Sombor Energy, Elliptic Sombor Characteristic Polynomial,  Eigenvalues,  Regular Graphs.

\medskip
\noindent{\bf AMS Subj.\ Class.:} 05C12, 05C50.

\section{Introduction}

 Let $G=(V,E)$ be a simple graph, with vertex set $V(G) = \{v_1, v_2,\ldots, v_n\}$. If two vertices $v_i$ and $v_j$ of $G$ are adjacent, then we use the notation $v_i \sim v_j $. For $v_i \in V(G)$, the degree of the vertex $v_i$, denoted by $d_i$, is the number of the vertices adjacent to $v_i$.

Let $A(G)$ be adjacency matrix of $G$ and $\lambda_1,\lambda_2,\ldots,\lambda_n$ its eigenvalues. These are said to be the eigenvalues of the graph $G$ and to form its spectrum \cite{Cve}. The energy $E(G)$ of the graph $G$ is defined as the sum of the absolute values of its eigenvalues
$$E(G)=\sum_{i=1}^n\vert\lambda_i\vert.$$
Details and more information on graph energy can be found in \cite{Gut,Gut1,Gut2,Maj}. There are many kinds of graph energies, such as Randi\'{c} energy \cite{Alikhani,Boz,Boz1,Das,GutB}, distance energy \cite{Ste}, incidence energy \cite{Boz2},  matching energy \cite{chen,Jis} and Laplacian energy \cite{Das0}. 

Sombor index is defined as  
$SO(G) =\sum_{uv\in E(G)}\sqrt{d_u^2+d_v^2}$ (see \cite{Gutman2}). More details on Sombor index can be found in {\cite{Alikhani1,chen0,AMC,Symmetry,Deng,Sombor,Li1,Red,Wang}}. Recently, in \cite{Gutman3}, Gutman introduced {Sombor} matrix of a graph $G$ as $A_{SO}(G)=(r_{ij})_{n\times n}$, where

\begin{displaymath}
 r_{ij}= \left\{ \begin{array}{ll}
\sqrt{d_i^2+d_j^2} & \textrm{if $v_i \sim v_j$}\\
0 & \textrm{otherwise.}
\end{array} \right.
\end{displaymath}

The eigenvalues of $A_{SO}(G)$ are denoted by $\rho_1\geq \rho_2\geq \ldots\geq \rho_n$, and are said to form the Sombor spectrum of the graph $G$. The {Sombor} characteristic polynomial $\phi _{SO}(G,\lambda)$ is 

$$\phi _{SO}(G,\lambda)=det(\lambda I -A_{SO}(G))=\prod_{i=1}^n (\lambda-\rho_i),$$

and Sombor energy ${E_{SO}}(G)$ is 

$${E_{SO}}(G)=\sum_{i=1}^{n}|\rho_i|.$$

{We refer the reader to \cite{Somborenergy,Gowtham2,Gut3,Gowtham} for more details on Sombor energy.}

In a recent paper (see \cite{Espinal}), the elliptic Sombor index of $G$ is defined as

$$ESO(G) =\sum_{uv\in E(G)}(d_u+d_v)\sqrt{d_u^2+d_v^2}.$$

Motivated by the definition of the Sombor matrix, we define the elliptic Sombor matrix as $A_{ESO}(G)=(r_{ij})_{n\times n}$, and

\begin{displaymath}
 r_{ij}= \left\{ \begin{array}{ll}
(d_i+d_j)\sqrt{d_i^2+d_j^2} & \textrm{if $v_i \sim v_j$}\\
0 & \textrm{otherwise.}
\end{array} \right.
\end{displaymath}

The eigenvalues of $A_{ESO}(G)$ are denoted by $\lambda_1\geq \lambda_2\geq \ldots\geq \lambda_n$, and are said to form the elliptic Sombor spectrum of the graph $G$. The  elliptic {Sombor} characteristic polynomial $\phi _{ESO}(G,\lambda)$ is 

$$\phi _{ESO}(G,\lambda)=det(\lambda I -A_{ESO}(G))=\prod_{i=1}^n (\lambda-\lambda_i),$$

and elliptic Sombor energy ${E_{ESO}}(G)$ is 

$${E_{ESO}}(G)=\sum_{i=1}^{n}|\lambda_i|.$$

Two graphs $G$ and $H$ are said to be {\it elliptic Sombor energy equivalent},
or simply ${\cal {E_{ESO}}}$-equivalent, written $G\sim H$, if
${E_{ESO}}(G)={E_{ESO}}(H)$. It is evident that the relation $\sim$ of being
${\cal {E_{ESO}}}$-equivalence
 is an equivalence relation on the family ${\cal G}$ of graphs, and thus ${\cal G}$ is partitioned into equivalence classes,
called the {\it ${\cal {E_{ESO}}}$-equivalent}. Given $G\in {\cal G}$, let
\[
[G]=\{H\in {\cal G}:H\sim G\}.
\]
We call $[G]$ the equivalence class determined by $G$.
A graph $G$ is said to be {\it elliptic Sombor energy unique}, or simply {\it ${\cal { E_{ESO}}}$-unique}, if $[G]=\{G\}$.

 A graph $G$ is called {\it $k$-regular} if all
vertices  have the same degree $k$.  One of the famous graphs is the Petersen
graph which is a symmetric non-planar $3$-regular graph of order $10$.  

There are exactly twenty one $3$-regular graphs of order 10 \cite{reza}.
In the study of elliptic Sombor energy, it is natural  to investigate
the elliptic Sombor  characteristic polynomial and elliptic Sombor energy of cubic graphs of order $10$ and check whether they  recognized by their  elliptic Sombor energy among other $3$-regular graphs with the same order. 
 We denote the Petersen graph by $P$. 

\medskip

In the next section we compute the elliptic Sombor energy of specific graphs. In Section 3, we study the elliptic Sombor energy of cubic graphs of order $10$. As a consequence we show that the Petersen graph cannot be determined by its elliptic Sombor energy.

\section{Elliptic Sombor energy of specific graphs}\label{sec1}
In this section, we study the elliptic Sombor characteristic polynomial and the elliptic Sombor energy for certain graphs. 
Here we  compute the elliptic Sombor characteristic polynomial of paths and cycles.

\begin{theorem}\label{thm-path scp}
For every $n\geq 5$, the elliptic Sombor characteristic polynomial of the path graph $P_n$ satisfy:
$$\phi _{ESO}(P_n,\lambda)=  \lambda ^2 \Lambda _{n-2}-90\lambda \Lambda _{n-3}+2025\Lambda _{n-4},$$
where for every $k\geq 3$, $\Lambda _k = \lambda \Lambda _{k-1}-128\Lambda _{k-2}$ with $\Lambda _1=\lambda $ and $\Lambda _2= \lambda ^2 -8$. Also the characteristic polynomial of $P_2,P_3$ and $P_4$ are $\lambda ^2 -8,\lambda ^3 -90 \lambda$ and $\lambda ^4 -218 \lambda ^2 +2025$ respectively.
\end{theorem}

\begin{proof}
	It is easy to see that the characteristic polynomial of $P_2$ is $\lambda ^2 -8$, Also for  $P_3$ is $\lambda ^3 -90 \lambda$ and for $P_4$ is $\lambda ^4 -218 \lambda ^2 +2025$. Now for every $k\geq 3$ consider

	\[
    	M_k:=\left(
         	\begin{array}{cccccc}
         	\lambda & -4\sqrt{8}  & 0  & \ldots  & 0 &   0         \\
			-4\sqrt{8} & \lambda & -4\sqrt{8}  & \ldots  &  0&   0          \\
         	0& -4\sqrt{8} & \lambda  & \ldots  & 0 &   0          \\
         \vdots	& \vdots & \vdots  & \ddots  & \vdots &  \vdots           \\					
			0&  0& 0  & \ldots     & \lambda &   -4\sqrt{8}          \\
			0& 0 & 0&  \ldots    & -4\sqrt{8} & \lambda            
        	\end{array}
    	\right)_{k\times k},
	\]
	
	and let $\Lambda _k= det(M_k)$. One can easily check that $\Lambda _k = \lambda \Lambda _{k-1}-128\Lambda _{k-2}$ . Now consider the path graph $P_n$. Suppose that $\phi _{ESO}(P_n,\lambda)= det (\lambda I - A_{ESO}(P_n))$. We have 
	
	\[
    	\phi _{ESO}(P_n,\lambda)= det\left(
         	\begin{array}{c|cccccc|c}
         	\lambda & -3\sqrt{5}  & 0 & 0 & \ldots & 0 & 0 &   0         \\
         	\hline
			-3\sqrt{5} & & & & & & &   0          \\
         	0&  & & & & &  &  0           \\
			0&   & &  & & &  &  0        \\
         \vdots	&  & &  & M_{n-2} & & &  \vdots           \\
			0&   & & & & &   &    0        \\						
			0&  & & & & &  &   -3\sqrt{5}          \\
			\hline
			0& 0 & 0& 0 & \ldots   & 0 & -3\sqrt{5} & \lambda           
        	\end{array}
    	\right)_{n\times n}.
	\]
	
	So,

\begin{align*}
    	\phi _{ESO}(P_n,\lambda)&= 	
    	 \lambda det\left(
         	\begin{array}{cccc|c}
			   & & &  &  0        \\
            &  & M_{n-2} & &  \vdots           \\
			    & & &   &    0        \\						
			 & & &  &   -3\sqrt{5}          \\
			\hline
			 0 & \ldots   & 0 & -3\sqrt{5} & \lambda           
        	\end{array}
    	\right)\\
    	& \quad+
    	3\sqrt{5} det\left(
         	\begin{array}{c|ccc|c}
			 - 3\sqrt{5} &  - 4\sqrt{8}  & \ldots  & 0 &  0        \\
			 \hline
          0  &  &  & &  0           \\
			  \vdots  & &M_{n-3} &   &    \vdots       \\						
			0 & & &  &   -3\sqrt{5}          \\
			\hline
			 0 & 0  & \ldots  & -3\sqrt{5} & \lambda           
        	\end{array}
        	\right).        	
\end{align*}

And so,

\begin{align*}
    	\phi_{ESO}(P_n,\lambda)&= 	
    	 \lambda\left(\lambda \Lambda _{n-2} + 3\sqrt{5} det\left(
         	\begin{array}{cccc|c}
			   & & &  &  0        \\
            &  & M_{n-3} & &  \vdots           \\
			    & & &   &    0        \\						
			 & & &  &   0         \\
			\hline
			 0 & \ldots   & 0 & -4\sqrt{8} & -3\sqrt{5}           
        	\end{array}
    	\right)\right)\\
    	& \quad -45
    	det\left(
         	\begin{array}{cccc|c}
			   & & &  &  0        \\
            &  & M_{n-3} & &  \vdots           \\
			    & & &   &    0        \\						
			 & & &  &   -3\sqrt{5}          \\
			\hline
			 0 & \ldots   & 0 & -3\sqrt{5} & \lambda           
        	\end{array}
    	\right)     	
\end{align*}

Hence,

\begin{align*}
    	\phi_{ESO}(P_n,\lambda)&= 	
    	 \lambda\left(\lambda \Lambda _{n-2} -45 \Lambda _{n-3}\right)\\
    	& \quad -45 \left(
    	\lambda \Lambda _{n-3} +3\sqrt{5} det\left(
         	\begin{array}{cccc|c}
			   & & &  &  0        \\
            &  & M_{n-4} & &  \vdots           \\
			    & & &   &    0        \\						
			 & & &  &   0         \\
			\hline
			 0 & \ldots   & 0 & -4\sqrt{8} & -3\sqrt{5}           
        	\end{array}
    	\right)\right) \\
    	&= 	
    	 \lambda\left(\lambda \Lambda _{n-2} -45 \Lambda _{n-3}\right)      -45 \left(
    	\lambda \Lambda _{n-3} -45\Lambda _{n-4} \right),	   	
\end{align*}

and therefore we have the result.\qed
\end{proof}

\begin{theorem}\label{thm-cycle}
For every $n\geq 3$, the elliptic Sombor characteristic polynomial of the cycle graph $C_n$ satisfy:
$$\phi _{ESO}(C_n,\lambda)= \lambda \Lambda _{n-1} -1024\Lambda _{n-2} +((-1)^{n+1})(2)(8\sqrt{8})^n ,$$
where for every $k\geq 3$, $\Lambda _k = \lambda \Lambda _{k-1}-8\Lambda _{k-2}$ with $\Lambda _1=\lambda $ and $\Lambda _2= \lambda ^2 -512$.
\end{theorem}

\begin{proof}
Similar to the proof of Theorem \ref{thm-path scp}, for every $k\geq3$, we consider

	\[
    	M_k:=\left(
         	\begin{array}{cccccccc}
         	\lambda & -8\sqrt{8}  & 0 & 0 & \ldots & 0 & 0 &   0         \\
			-8\sqrt{8} & \lambda & -8\sqrt{8} & 0 & \ldots & 0 &  0&   0          \\
         	0& -8\sqrt{8} & \lambda & -8\sqrt{8} & \ldots & 0 & 0 &   0          \\
			0& 0 & -8\sqrt{8} & \lambda &  \ldots &0  & 0  &  0        \\
         \vdots	& \vdots & \vdots & \vdots & \ddots & \vdots & \vdots &  \vdots           \\
			0&  0&0 & 0& \ldots   & \lambda & -8\sqrt{8}  &    0        \\						
			0&  0& 0 &0 & \ldots   & -8\sqrt{8} & \lambda &   -8\sqrt{8}          \\
			0& 0 & 0& 0 & \ldots   & 0 & -8\sqrt{8} & \lambda            
        	\end{array}
    	\right)_{k\times k},
	\]
	
	and let $\Lambda _k= det(M_k)$. We have $\Lambda _k = \lambda \Lambda _{k-1}-512\Lambda _{k-2}$. Suppose that $\phi _{ESO}(C_n,\lambda)= det (\lambda I - A_{ESO}(C_n))$. We have 
	
	\[
    	\phi _{ESO}(C_n,\lambda)= det\left(
         	\begin{array}{c|ccccccc}
         	\lambda & -8\sqrt{8}  & 0 & 0 & \ldots & 0 & 0 &   -8\sqrt{8}         \\
         	\hline
			-8\sqrt{8} & & & & & & &             \\
         	0&  & & & & &  &             \\
			0&   & &  & & &  &          \\
         \vdots	&  & &  & M_{n-1} & & &             \\
			0&   & & & & &   &          \\						
			0&  & & & & &  &            \\
			-8\sqrt{8}&  & &  &   &  &  &           
        	\end{array}
    	\right)_{n\times n}.
	\]
	
	So,

\begin{align*}
    	\phi _{ESO}(P_n,\lambda)&= 	\lambda \Lambda _{n-1}+
    	 8\sqrt{8} det\left(
         	\begin{array}{c|cccc}
			-8\sqrt{8}   & -8\sqrt{8} & 0 & \ldots &  0        \\
			   \hline
           0 &  &  & &             \\
			   \vdots & & M_{n-2} &   &            \\						
			0 & & &  &            \\
			 -8\sqrt{8} &   &  &  &            
        	\end{array}
    	\right)\\
    	& \quad+
    	(-1)^{n+1}(-8\sqrt{8}) det\left(
         	\begin{array}{c|cccc}
			 - 8\sqrt{8} &    &   &  &         \\
          0  &  &  & M_{n-2} &             \\
			  \vdots  & & &   &        \\						
			0 & & &  &             \\
			\hline
			 -8\sqrt{8}   & 0  & \ldots  & 0 &   -8\sqrt{8}       
        	\end{array}
        	\right).        	
\end{align*}

Hence,

\begin{align*}
    	\phi_{ESO}(C_n,\lambda)&= 	\lambda \Lambda _{n-1}+
    	 8\sqrt{8}\left( -8\sqrt{8}\Lambda _{n-2} + (-1)^{n}(-8\sqrt{8})^{n-1}
    	 \right)\\
    	& \quad +(-1)^{n+1}(-8\sqrt{8})
    	\left( (-8\sqrt{8})^{n-1}+(-1)^n(-8\sqrt{8})\Lambda _{n-2}
    	\right),     	
\end{align*}

and therefore we have the result.\qed
\end{proof}

Now we consider to star graph $S_n$ and find its elliptic Sombor characteristic polynomial and elliptic Sombor energy. We need the following Lemma.

\begin{lemma} \label{new}\rm\cite{Cve}
If $M$ is a nonsingular square matrix, then
$$det\left(  \begin{array}{cc}
M&N  \\
P& Q \\
\end{array}\right)=det (M) det( Q-PM^{-1}N).
$$
\end{lemma}

\begin{theorem}
For $n\geq 2$,
\begin{itemize}
\item[(i)] The elliptic Sombor characteristic polynomial of the star graph $S_n=K_{1,n-1}$ is
$$\phi_{ESO}(S_n,\lambda))=\lambda^{n-2}\left(\lambda ^2 -(n-1)(n^4-2n^3+2n^2)\right).$$
\item[(ii)] The elliptic Sombor energy of $S_n$ is
$${E_{ESO}}(S_n)=2n\sqrt{(n-1)(n^2-2n+2)}.$$
\end{itemize}
\end{theorem}

\begin{proof}
\begin{enumerate}
\item[(i)] One can easily check that the elliptic Sombor matrix of $K_{1,n-1}$ is
$$A_{ESO}(S_n)=n\sqrt{n^2-2n+2}\left( \begin{array}{cc}
0_{1\times 1}&J_{1\times {n-1}} \\
J_{{n-1}\times 1}&0_{{n-1}\times {n-1}}   \\
\end{array} \right).$$
We  have $det(\lambda I - A_{ESO}(S_n))=$
$$det
\left(  \begin{array}{cc}
\lambda  & -n\sqrt{n^2-2n+2}J_{1\times (n-1)}  \\
-n\sqrt{n^2-2n+2}J_{(n-1)\times 1}& \lambda I_{n-1} \\
\end{array}\right).
$$
Using Lemma \ref{new}, $det(\lambda I -A_{ESO}(S_n))=$
$$
\lambda det( \lambda I_{n-1} - n\sqrt{n^2-2n+2}J_{(n-1)\times 1}\frac{1}{\lambda} (n\sqrt{n^2-2n+2}J_{1\times (n-1)})).
$$

Since $J_{(n-1)\times 1}J_{1\times (n-1)}=J_{n-1}$, so
\begin{align*}
det(\lambda I - A_{ESO}(S_n))&=\lambda det( \lambda I_{n-1} - \frac{n^2}{\lambda}(n^2-2n+2)J_{n-1})\\
&=\lambda ^{2-n}
det( \lambda ^2 I_{n-1} - n^2(n^2-2n+2)J_{n-1}).
\end{align*}
On the other hand, the eigenvalues of $J_{n-1}$ are $n-1$ (once) and 0 ($n-2$ times),  the eigenvalues of $n^2(n^2-2n+2)J_{n-1}$ are $(n-1)(n^2)(n^2-2n+2)$ (once) and 0 ($n-2$ times). Therefore 
$$\phi_{ESO}(S_n,\lambda))=\lambda^{n-2}\left(\lambda ^2 -(n-1)(n^4-2n^3+2n^2)\right).$$

 \item[(ii)] It follows from Part (i).\quad\qed

 \end{enumerate}
\end{proof}

We close this section by computing the  elliptic Sombor characteristic polynomial of complete bipartite graphs and their Sombor energy.

\begin{theorem}\label{thm-bipartite}
For natural number $m,n\neq 1$,
\begin{itemize}
\item[(i)] The elliptic {Sombor} characteristic polynomial of complete bipartite graph $K_{m,n}$ is
$$\phi_{ESO}(K_{m,n},\lambda)=\lambda^{m+n-2}(\lambda^2 -mn(m^2+n^2)).$$
\item[(ii)] The elliptic {Sombor} energy of $K_{m,n}$ is
$2(m+n)\sqrt{mn(m^2+n^2)}$.
\end{itemize}
\end{theorem}

\begin{proof}
\begin{enumerate}
\item[(i)]
It is easy to see that the elliptic Sombor matrix of $K_{m,n}$ is
$$(m+n)\sqrt{m^2+n^2}\left( \begin{array}{cc}
0_{m\times m}&J_{m\times n} \\
J_{n\times m}&0_{n\times n}   \\
\end{array} \right).$$
Using Lemma \ref{new} we have 

$$det(\lambda I -A_{ESO}(K_{m,n}))=det
\left(  \begin{array}{cc}
\lambda I_m & -(m+n)\sqrt{m^2+n^2}J_{m\times n}  \\
-(m+n)\sqrt{m^2+n^2}J_{n\times m}& \lambda I_n \\
\end{array}\right).
$$
So $det(\lambda I -A_{ESO}(K_{m,n}))=$
$$
det (\lambda I_m) det( \lambda I_n - (m+n)\sqrt{m^2+n^2}J_{n\times m}\frac{1}{\lambda}I_m (m+n)\sqrt{m^2+n^2}J_{m\times n}).
$$

We know that $J_{n\times m}J_{m\times n}=mJ_n$. Therefore
\begin{align*}
det(\lambda I -A_{SO}(K_{m,n}))&=\lambda ^m det( \lambda I_n - \frac{1}{\lambda }m(m^2+n^2)(m+n)^2J_n)\\
&=\lambda ^{m-n}
det( \lambda ^2 I_n - m(m^2+n^2)(m+n)^2J_n).
\end{align*}
The eigenvalues of $J_n$ are $n$ (once) and 0 ($n-1$ times). So the eigenvalues of $m(m^2+n^2)(m+n)^2J_n$ are $mn(m^2+n^2)(m+n)^2$ (once) and 0 ($n-1$ times).
Hence
$$\phi_{SO}(K_{m,n},\lambda)=\lambda^{m+n-2}(\lambda^2 -mn(m^2+n^2)(m+n)^2).$$

\item[(ii)] It follows from Part (i).\quad\qed
\end{enumerate}

\end{proof}

\section{Elliptic Sombor energy of $k$-regular graphs}

In this section  we consider $2$-regular and $3$-regular graphs. As a beginning of this section,  we have the following easy lemma:

\begin{lemma}\label{union}
Let $G=G_1\cup G_2\cup G_3\cup \ldots\cup G_n$. Then
\begin{enumerate}
\item[(i)] $\phi_{ESO}(G)=\prod_{i=1}^{n}\phi_{ESO}(G_i)$.

\item[(ii)] ${E_{ESO}}(G)=\sum_{i=1}^{n}{E_{ESO}}(G_i)$.
\end{enumerate}
\end{lemma}

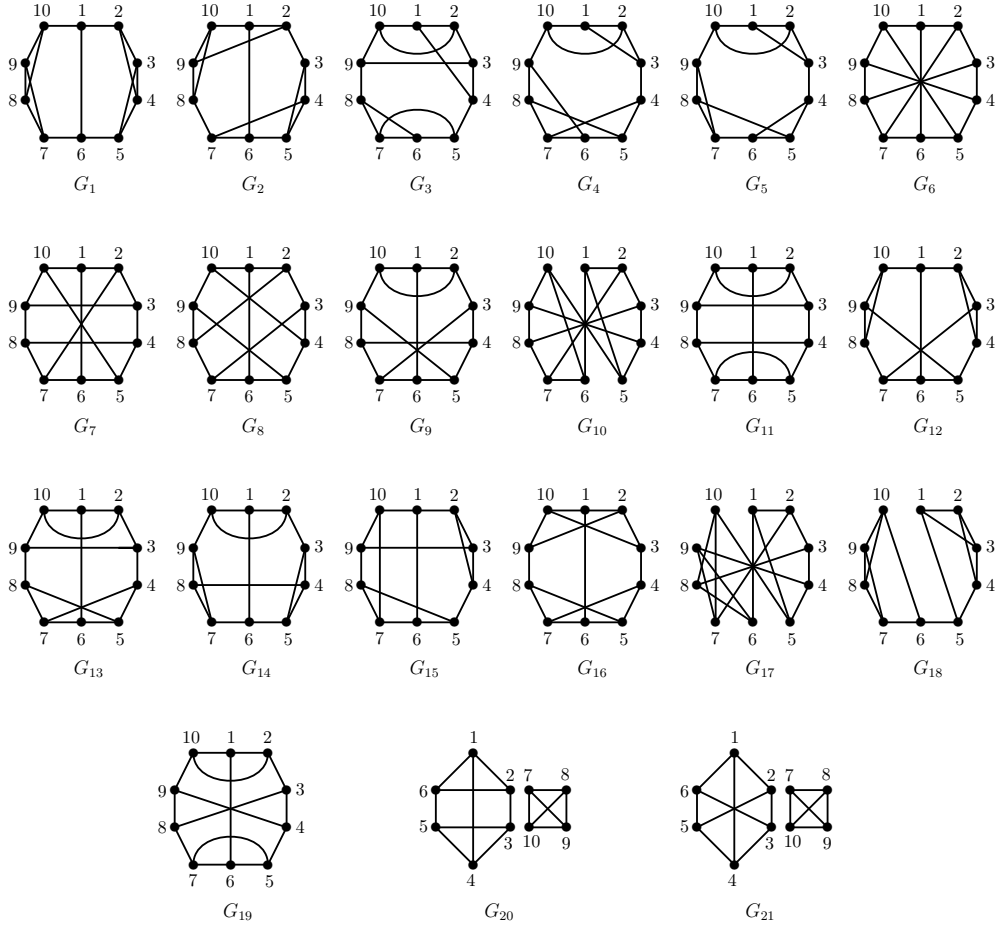
\begin{figure}[!h]
\begin{center}
\psscalebox{0.62 0.62}
{
\begin{pspicture}(0,-10.145)(21.18,9.525)
\rput[bl](1.48,9.255){1}
\rput[bl](2.26,9.235){2}
\rput[bl](2.96,8.135){3}
\rput[bl](2.3,6.155){5}
\rput[bl](1.46,6.155){6}
\rput[bl](2.96,7.315){4}
\rput[bl](0.66,6.175){7}
\rput[bl](0.0,7.315){8}
\rput[bl](1.4,5.455){\large{$G_1$}}
\rput[bl](0.0,8.095){9}
\rput[bl](0.52,9.255){10}
\rput[bl](5.0,5.455){\large{$G_2$}}
\rput[bl](8.6,5.455){\large{$G_3$}}
\rput[bl](12.2,5.455){\large{$G_4$}}
\rput[bl](15.8,5.455){\large{$G_5$}}
\rput[bl](19.4,5.455){\large{$G_6$}}
\psdots[linecolor=black, dotsize=0.2](0.76,9.035)
\psdots[linecolor=black, dotsize=0.2](1.56,9.035)
\psdots[linecolor=black, dotsize=0.2](2.36,9.035)
\psdots[linecolor=black, dotsize=0.2](2.76,8.235)
\psdots[linecolor=black, dotsize=0.2](2.76,7.435)
\psdots[linecolor=black, dotsize=0.2](2.36,6.635)
\psdots[linecolor=black, dotsize=0.2](1.56,6.635)
\psdots[linecolor=black, dotsize=0.2](0.76,6.635)
\psdots[linecolor=black, dotsize=0.2](0.36,7.435)
\psdots[linecolor=black, dotsize=0.2](0.36,8.235)
\psline[linecolor=black, linewidth=0.04](0.76,9.035)(0.36,8.235)(0.36,7.435)(0.76,6.635)(0.76,6.635)
\psline[linecolor=black, linewidth=0.04](2.36,9.035)(2.76,8.235)(2.76,7.435)(2.36,6.635)(2.36,6.635)
\psline[linecolor=black, linewidth=0.04](0.76,9.035)(2.36,9.035)(2.36,9.035)
\psline[linecolor=black, linewidth=0.04](0.76,6.635)(2.36,6.635)(2.36,6.635)
\rput[bl](5.08,9.255){1}
\rput[bl](5.86,9.235){2}
\rput[bl](6.56,8.135){3}
\rput[bl](5.9,6.155){5}
\rput[bl](5.06,6.155){6}
\rput[bl](6.56,7.315){4}
\rput[bl](4.26,6.175){7}
\rput[bl](3.6,7.315){8}
\rput[bl](3.6,8.095){9}
\rput[bl](4.12,9.255){10}
\psdots[linecolor=black, dotsize=0.2](4.36,9.035)
\psdots[linecolor=black, dotsize=0.2](5.16,9.035)
\psdots[linecolor=black, dotsize=0.2](5.96,9.035)
\psdots[linecolor=black, dotsize=0.2](6.36,8.235)
\psdots[linecolor=black, dotsize=0.2](6.36,7.435)
\psdots[linecolor=black, dotsize=0.2](5.96,6.635)
\psdots[linecolor=black, dotsize=0.2](5.16,6.635)
\psdots[linecolor=black, dotsize=0.2](4.36,6.635)
\psdots[linecolor=black, dotsize=0.2](3.96,7.435)
\psdots[linecolor=black, dotsize=0.2](3.96,8.235)
\psline[linecolor=black, linewidth=0.04](4.36,9.035)(3.96,8.235)(3.96,7.435)(4.36,6.635)(4.36,6.635)
\psline[linecolor=black, linewidth=0.04](5.96,9.035)(6.36,8.235)(6.36,7.435)(5.96,6.635)(5.96,6.635)
\psline[linecolor=black, linewidth=0.04](4.36,9.035)(5.96,9.035)(5.96,9.035)
\psline[linecolor=black, linewidth=0.04](4.36,6.635)(5.96,6.635)(5.96,6.635)
\rput[bl](8.68,9.255){1}
\rput[bl](9.46,9.235){2}
\rput[bl](10.16,8.135){3}
\rput[bl](9.5,6.155){5}
\rput[bl](8.66,6.155){6}
\rput[bl](10.16,7.315){4}
\rput[bl](7.86,6.175){7}
\rput[bl](7.2,7.315){8}
\rput[bl](7.2,8.095){9}
\rput[bl](7.72,9.255){10}
\psdots[linecolor=black, dotsize=0.2](7.96,9.035)
\psdots[linecolor=black, dotsize=0.2](8.76,9.035)
\psdots[linecolor=black, dotsize=0.2](9.56,9.035)
\psdots[linecolor=black, dotsize=0.2](9.96,8.235)
\psdots[linecolor=black, dotsize=0.2](9.96,7.435)
\psdots[linecolor=black, dotsize=0.2](9.56,6.635)
\psdots[linecolor=black, dotsize=0.2](8.76,6.635)
\psdots[linecolor=black, dotsize=0.2](7.96,6.635)
\psdots[linecolor=black, dotsize=0.2](7.56,7.435)
\psdots[linecolor=black, dotsize=0.2](7.56,8.235)
\psline[linecolor=black, linewidth=0.04](7.96,9.035)(7.56,8.235)(7.56,7.435)(7.96,6.635)(7.96,6.635)
\psline[linecolor=black, linewidth=0.04](9.56,9.035)(9.96,8.235)(9.96,7.435)(9.56,6.635)(9.56,6.635)
\psline[linecolor=black, linewidth=0.04](7.96,9.035)(9.56,9.035)(9.56,9.035)
\psline[linecolor=black, linewidth=0.04](7.96,6.635)(9.56,6.635)(9.56,6.635)
\rput[bl](12.28,9.255){1}
\rput[bl](13.06,9.235){2}
\rput[bl](13.76,8.135){3}
\rput[bl](13.1,6.155){5}
\rput[bl](12.26,6.155){6}
\rput[bl](13.76,7.315){4}
\rput[bl](11.46,6.175){7}
\rput[bl](10.8,7.315){8}
\rput[bl](10.8,8.095){9}
\rput[bl](11.32,9.255){10}
\psdots[linecolor=black, dotsize=0.2](11.56,9.035)
\psdots[linecolor=black, dotsize=0.2](12.36,9.035)
\psdots[linecolor=black, dotsize=0.2](13.16,9.035)
\psdots[linecolor=black, dotsize=0.2](13.56,8.235)
\psdots[linecolor=black, dotsize=0.2](13.56,7.435)
\psdots[linecolor=black, dotsize=0.2](13.16,6.635)
\psdots[linecolor=black, dotsize=0.2](12.36,6.635)
\psdots[linecolor=black, dotsize=0.2](11.56,6.635)
\psdots[linecolor=black, dotsize=0.2](11.16,7.435)
\psdots[linecolor=black, dotsize=0.2](11.16,8.235)
\psline[linecolor=black, linewidth=0.04](11.56,9.035)(11.16,8.235)(11.16,7.435)(11.56,6.635)(11.56,6.635)
\psline[linecolor=black, linewidth=0.04](13.16,9.035)(13.56,8.235)(13.56,7.435)(13.16,6.635)(13.16,6.635)
\psline[linecolor=black, linewidth=0.04](11.56,9.035)(13.16,9.035)(13.16,9.035)
\psline[linecolor=black, linewidth=0.04](11.56,6.635)(13.16,6.635)(13.16,6.635)
\rput[bl](15.88,9.255){1}
\rput[bl](16.66,9.235){2}
\rput[bl](17.36,8.135){3}
\rput[bl](16.7,6.155){5}
\rput[bl](15.86,6.155){6}
\rput[bl](17.36,7.315){4}
\rput[bl](15.06,6.175){7}
\rput[bl](14.4,7.315){8}
\rput[bl](14.4,8.095){9}
\rput[bl](14.92,9.255){10}
\psdots[linecolor=black, dotsize=0.2](15.16,9.035)
\psdots[linecolor=black, dotsize=0.2](15.96,9.035)
\psdots[linecolor=black, dotsize=0.2](16.76,9.035)
\psdots[linecolor=black, dotsize=0.2](17.16,8.235)
\psdots[linecolor=black, dotsize=0.2](17.16,7.435)
\psdots[linecolor=black, dotsize=0.2](16.76,6.635)
\psdots[linecolor=black, dotsize=0.2](15.96,6.635)
\psdots[linecolor=black, dotsize=0.2](15.16,6.635)
\psdots[linecolor=black, dotsize=0.2](14.76,7.435)
\psdots[linecolor=black, dotsize=0.2](14.76,8.235)
\psline[linecolor=black, linewidth=0.04](15.16,9.035)(14.76,8.235)(14.76,7.435)(15.16,6.635)(15.16,6.635)
\psline[linecolor=black, linewidth=0.04](16.76,9.035)(17.16,8.235)(17.16,7.435)(16.76,6.635)(16.76,6.635)
\psline[linecolor=black, linewidth=0.04](15.16,9.035)(16.76,9.035)(16.76,9.035)
\psline[linecolor=black, linewidth=0.04](15.16,6.635)(16.76,6.635)(16.76,6.635)
\rput[bl](19.48,9.255){1}
\rput[bl](20.26,9.235){2}
\rput[bl](20.96,8.135){3}
\rput[bl](20.3,6.155){5}
\rput[bl](19.46,6.155){6}
\rput[bl](20.96,7.315){4}
\rput[bl](18.66,6.175){7}
\rput[bl](18.0,7.315){8}
\rput[bl](18.0,8.095){9}
\rput[bl](18.52,9.255){10}
\psdots[linecolor=black, dotsize=0.2](18.76,9.035)
\psdots[linecolor=black, dotsize=0.2](19.56,9.035)
\psdots[linecolor=black, dotsize=0.2](20.36,9.035)
\psdots[linecolor=black, dotsize=0.2](20.76,8.235)
\psdots[linecolor=black, dotsize=0.2](20.76,7.435)
\psdots[linecolor=black, dotsize=0.2](20.36,6.635)
\psdots[linecolor=black, dotsize=0.2](19.56,6.635)
\psdots[linecolor=black, dotsize=0.2](18.76,6.635)
\psdots[linecolor=black, dotsize=0.2](18.36,7.435)
\psdots[linecolor=black, dotsize=0.2](18.36,8.235)
\psline[linecolor=black, linewidth=0.04](18.76,9.035)(18.36,8.235)(18.36,7.435)(18.76,6.635)(18.76,6.635)
\psline[linecolor=black, linewidth=0.04](20.36,9.035)(20.76,8.235)(20.76,7.435)(20.36,6.635)(20.36,6.635)
\psline[linecolor=black, linewidth=0.04](18.76,9.035)(20.36,9.035)(20.36,9.035)
\psline[linecolor=black, linewidth=0.04](18.76,6.635)(20.36,6.635)(20.36,6.635)
\rput[bl](1.48,4.055){1}
\rput[bl](2.26,4.035){2}
\rput[bl](2.96,2.935){3}
\rput[bl](2.3,0.955){5}
\rput[bl](1.46,0.955){6}
\rput[bl](2.96,2.115){4}
\rput[bl](0.66,0.975){7}
\rput[bl](0.0,2.115){8}
\rput[bl](0.0,2.895){9}
\rput[bl](0.52,4.055){10}
\psdots[linecolor=black, dotsize=0.2](0.76,3.835)
\psdots[linecolor=black, dotsize=0.2](1.56,3.835)
\psdots[linecolor=black, dotsize=0.2](2.36,3.835)
\psdots[linecolor=black, dotsize=0.2](2.76,3.035)
\psdots[linecolor=black, dotsize=0.2](2.76,2.235)
\psdots[linecolor=black, dotsize=0.2](2.36,1.435)
\psdots[linecolor=black, dotsize=0.2](1.56,1.435)
\psdots[linecolor=black, dotsize=0.2](0.76,1.435)
\psdots[linecolor=black, dotsize=0.2](0.36,2.235)
\psdots[linecolor=black, dotsize=0.2](0.36,3.035)
\psline[linecolor=black, linewidth=0.04](0.76,3.835)(0.36,3.035)(0.36,2.235)(0.76,1.435)(0.76,1.435)
\psline[linecolor=black, linewidth=0.04](2.36,3.835)(2.76,3.035)(2.76,2.235)(2.36,1.435)(2.36,1.435)
\psline[linecolor=black, linewidth=0.04](0.76,3.835)(2.36,3.835)(2.36,3.835)
\psline[linecolor=black, linewidth=0.04](0.76,1.435)(2.36,1.435)(2.36,1.435)
\rput[bl](5.08,4.055){1}
\rput[bl](5.86,4.035){2}
\rput[bl](6.56,2.935){3}
\rput[bl](5.9,0.955){5}
\rput[bl](5.06,0.955){6}
\rput[bl](6.56,2.115){4}
\rput[bl](4.26,0.975){7}
\rput[bl](3.6,2.115){8}
\rput[bl](3.6,2.895){9}
\rput[bl](4.12,4.055){10}
\psdots[linecolor=black, dotsize=0.2](4.36,3.835)
\psdots[linecolor=black, dotsize=0.2](5.16,3.835)
\psdots[linecolor=black, dotsize=0.2](5.96,3.835)
\psdots[linecolor=black, dotsize=0.2](6.36,3.035)
\psdots[linecolor=black, dotsize=0.2](6.36,2.235)
\psdots[linecolor=black, dotsize=0.2](5.96,1.435)
\psdots[linecolor=black, dotsize=0.2](5.16,1.435)
\psdots[linecolor=black, dotsize=0.2](4.36,1.435)
\psdots[linecolor=black, dotsize=0.2](3.96,2.235)
\psdots[linecolor=black, dotsize=0.2](3.96,3.035)
\psline[linecolor=black, linewidth=0.04](4.36,3.835)(3.96,3.035)(3.96,2.235)(4.36,1.435)(4.36,1.435)
\psline[linecolor=black, linewidth=0.04](5.96,3.835)(6.36,3.035)(6.36,2.235)(5.96,1.435)(5.96,1.435)
\psline[linecolor=black, linewidth=0.04](4.36,3.835)(5.96,3.835)(5.96,3.835)
\psline[linecolor=black, linewidth=0.04](4.36,1.435)(5.96,1.435)(5.96,1.435)
\rput[bl](8.68,4.055){1}
\rput[bl](9.46,4.035){2}
\rput[bl](10.16,2.935){3}
\rput[bl](9.5,0.955){5}
\rput[bl](8.66,0.955){6}
\rput[bl](10.16,2.115){4}
\rput[bl](7.86,0.975){7}
\rput[bl](7.2,2.115){8}
\rput[bl](7.2,2.895){9}
\rput[bl](7.72,4.055){10}
\psdots[linecolor=black, dotsize=0.2](7.96,3.835)
\psdots[linecolor=black, dotsize=0.2](8.76,3.835)
\psdots[linecolor=black, dotsize=0.2](9.56,3.835)
\psdots[linecolor=black, dotsize=0.2](9.96,3.035)
\psdots[linecolor=black, dotsize=0.2](9.96,2.235)
\psdots[linecolor=black, dotsize=0.2](9.56,1.435)
\psdots[linecolor=black, dotsize=0.2](8.76,1.435)
\psdots[linecolor=black, dotsize=0.2](7.96,1.435)
\psdots[linecolor=black, dotsize=0.2](7.56,2.235)
\psdots[linecolor=black, dotsize=0.2](7.56,3.035)
\psline[linecolor=black, linewidth=0.04](7.96,3.835)(7.56,3.035)(7.56,2.235)(7.96,1.435)(7.96,1.435)
\psline[linecolor=black, linewidth=0.04](9.56,3.835)(9.96,3.035)(9.96,2.235)(9.56,1.435)(9.56,1.435)
\psline[linecolor=black, linewidth=0.04](7.96,3.835)(9.56,3.835)(9.56,3.835)
\psline[linecolor=black, linewidth=0.04](7.96,1.435)(9.56,1.435)(9.56,1.435)
\rput[bl](12.28,4.055){1}
\rput[bl](13.06,4.035){2}
\rput[bl](13.76,2.935){3}
\rput[bl](13.1,0.955){5}
\rput[bl](12.26,0.955){6}
\rput[bl](13.76,2.115){4}
\rput[bl](11.46,0.975){7}
\rput[bl](10.8,2.115){8}
\rput[bl](10.8,2.895){9}
\rput[bl](11.32,4.055){10}
\psdots[linecolor=black, dotsize=0.2](11.56,3.835)
\psdots[linecolor=black, dotsize=0.2](12.36,3.835)
\psdots[linecolor=black, dotsize=0.2](13.16,3.835)
\psdots[linecolor=black, dotsize=0.2](13.56,3.035)
\psdots[linecolor=black, dotsize=0.2](13.56,2.235)
\psdots[linecolor=black, dotsize=0.2](13.16,1.435)
\psdots[linecolor=black, dotsize=0.2](12.36,1.435)
\psdots[linecolor=black, dotsize=0.2](11.56,1.435)
\psdots[linecolor=black, dotsize=0.2](11.16,2.235)
\psdots[linecolor=black, dotsize=0.2](11.16,3.035)
\rput[bl](15.88,4.055){1}
\rput[bl](16.66,4.035){2}
\rput[bl](17.36,2.935){3}
\rput[bl](16.7,0.955){5}
\rput[bl](15.86,0.955){6}
\rput[bl](17.36,2.115){4}
\rput[bl](15.06,0.975){7}
\rput[bl](14.4,2.115){8}
\rput[bl](14.4,2.895){9}
\rput[bl](14.92,4.055){10}
\psdots[linecolor=black, dotsize=0.2](15.16,3.835)
\psdots[linecolor=black, dotsize=0.2](15.96,3.835)
\psdots[linecolor=black, dotsize=0.2](16.76,3.835)
\psdots[linecolor=black, dotsize=0.2](17.16,3.035)
\psdots[linecolor=black, dotsize=0.2](17.16,2.235)
\psdots[linecolor=black, dotsize=0.2](16.76,1.435)
\psdots[linecolor=black, dotsize=0.2](15.96,1.435)
\psdots[linecolor=black, dotsize=0.2](15.16,1.435)
\psdots[linecolor=black, dotsize=0.2](14.76,2.235)
\psdots[linecolor=black, dotsize=0.2](14.76,3.035)
\psline[linecolor=black, linewidth=0.04](15.16,3.835)(14.76,3.035)(14.76,2.235)(15.16,1.435)(15.16,1.435)
\psline[linecolor=black, linewidth=0.04](16.76,3.835)(17.16,3.035)(17.16,2.235)(16.76,1.435)(16.76,1.435)
\psline[linecolor=black, linewidth=0.04](15.16,3.835)(16.76,3.835)(16.76,3.835)
\psline[linecolor=black, linewidth=0.04](15.16,1.435)(16.76,1.435)(16.76,1.435)
\rput[bl](19.48,4.055){1}
\rput[bl](20.26,4.035){2}
\rput[bl](20.96,2.935){3}
\rput[bl](20.3,0.955){5}
\rput[bl](19.46,0.955){6}
\rput[bl](20.96,2.115){4}
\rput[bl](18.66,0.975){7}
\rput[bl](18.0,2.115){8}
\rput[bl](18.0,2.895){9}
\rput[bl](18.52,4.055){10}
\psdots[linecolor=black, dotsize=0.2](18.76,3.835)
\psdots[linecolor=black, dotsize=0.2](19.56,3.835)
\psdots[linecolor=black, dotsize=0.2](20.36,3.835)
\psdots[linecolor=black, dotsize=0.2](20.76,3.035)
\psdots[linecolor=black, dotsize=0.2](20.76,2.235)
\psdots[linecolor=black, dotsize=0.2](20.36,1.435)
\psdots[linecolor=black, dotsize=0.2](19.56,1.435)
\psdots[linecolor=black, dotsize=0.2](18.76,1.435)
\psdots[linecolor=black, dotsize=0.2](18.36,2.235)
\psdots[linecolor=black, dotsize=0.2](18.36,3.035)
\psline[linecolor=black, linewidth=0.04](18.76,3.835)(18.36,3.035)(18.36,2.235)(18.76,1.435)(18.76,1.435)
\psline[linecolor=black, linewidth=0.04](20.36,3.835)(20.76,3.035)(20.76,2.235)(20.36,1.435)(20.36,1.435)
\psline[linecolor=black, linewidth=0.04](18.76,3.835)(20.36,3.835)(20.36,3.835)
\psline[linecolor=black, linewidth=0.04](18.76,1.435)(20.36,1.435)(20.36,1.435)
\rput[bl](1.48,-1.145){1}
\rput[bl](2.26,-1.165){2}
\rput[bl](2.96,-2.265){3}
\rput[bl](2.3,-4.245){5}
\rput[bl](1.46,-4.245){6}
\rput[bl](2.96,-3.085){4}
\rput[bl](0.66,-4.225){7}
\rput[bl](0.0,-3.085){8}
\rput[bl](0.0,-2.305){9}
\rput[bl](0.52,-1.145){10}
\psdots[linecolor=black, dotsize=0.2](0.76,-1.365)
\psdots[linecolor=black, dotsize=0.2](1.56,-1.365)
\psdots[linecolor=black, dotsize=0.2](2.36,-1.365)
\psdots[linecolor=black, dotsize=0.2](2.76,-2.165)
\psdots[linecolor=black, dotsize=0.2](2.76,-2.965)
\psdots[linecolor=black, dotsize=0.2](2.36,-3.765)
\psdots[linecolor=black, dotsize=0.2](1.56,-3.765)
\psdots[linecolor=black, dotsize=0.2](0.76,-3.765)
\psdots[linecolor=black, dotsize=0.2](0.36,-2.965)
\psdots[linecolor=black, dotsize=0.2](0.36,-2.165)
\psline[linecolor=black, linewidth=0.04](0.76,-1.365)(0.36,-2.165)(0.36,-2.965)(0.76,-3.765)(0.76,-3.765)
\psline[linecolor=black, linewidth=0.04](2.36,-1.365)(2.76,-2.165)(2.76,-2.965)(2.36,-3.765)(2.36,-3.765)
\psline[linecolor=black, linewidth=0.04](0.76,-1.365)(2.36,-1.365)(2.36,-1.365)
\psline[linecolor=black, linewidth=0.04](0.76,-3.765)(2.36,-3.765)(2.36,-3.765)
\rput[bl](5.08,-1.145){1}
\rput[bl](5.86,-1.165){2}
\rput[bl](6.56,-2.265){3}
\rput[bl](5.9,-4.245){5}
\rput[bl](5.06,-4.245){6}
\rput[bl](6.56,-3.085){4}
\rput[bl](4.26,-4.225){7}
\rput[bl](3.6,-3.085){8}
\rput[bl](3.6,-2.305){9}
\rput[bl](4.12,-1.145){10}
\psdots[linecolor=black, dotsize=0.2](4.36,-1.365)
\psdots[linecolor=black, dotsize=0.2](5.16,-1.365)
\psdots[linecolor=black, dotsize=0.2](5.96,-1.365)
\psdots[linecolor=black, dotsize=0.2](6.36,-2.165)
\psdots[linecolor=black, dotsize=0.2](6.36,-2.965)
\psdots[linecolor=black, dotsize=0.2](5.96,-3.765)
\psdots[linecolor=black, dotsize=0.2](5.16,-3.765)
\psdots[linecolor=black, dotsize=0.2](4.36,-3.765)
\psdots[linecolor=black, dotsize=0.2](3.96,-2.965)
\psdots[linecolor=black, dotsize=0.2](3.96,-2.165)
\psline[linecolor=black, linewidth=0.04](4.36,-1.365)(3.96,-2.165)(3.96,-2.965)(4.36,-3.765)(4.36,-3.765)
\psline[linecolor=black, linewidth=0.04](5.96,-1.365)(6.36,-2.165)(6.36,-2.965)(5.96,-3.765)(5.96,-3.765)
\psline[linecolor=black, linewidth=0.04](4.36,-1.365)(5.96,-1.365)(5.96,-1.365)
\psline[linecolor=black, linewidth=0.04](4.36,-3.765)(5.96,-3.765)(5.96,-3.765)
\rput[bl](8.68,-1.145){1}
\rput[bl](9.46,-1.165){2}
\rput[bl](10.16,-2.265){3}
\rput[bl](9.5,-4.245){5}
\rput[bl](8.66,-4.245){6}
\rput[bl](10.16,-3.085){4}
\rput[bl](7.86,-4.225){7}
\rput[bl](7.2,-3.085){8}
\rput[bl](7.2,-2.305){9}
\rput[bl](7.72,-1.145){10}
\psdots[linecolor=black, dotsize=0.2](7.96,-1.365)
\psdots[linecolor=black, dotsize=0.2](8.76,-1.365)
\psdots[linecolor=black, dotsize=0.2](9.56,-1.365)
\psdots[linecolor=black, dotsize=0.2](9.96,-2.165)
\psdots[linecolor=black, dotsize=0.2](9.96,-2.965)
\psdots[linecolor=black, dotsize=0.2](9.56,-3.765)
\psdots[linecolor=black, dotsize=0.2](8.76,-3.765)
\psdots[linecolor=black, dotsize=0.2](7.96,-3.765)
\psdots[linecolor=black, dotsize=0.2](7.56,-2.965)
\psdots[linecolor=black, dotsize=0.2](7.56,-2.165)
\psline[linecolor=black, linewidth=0.04](7.96,-1.365)(7.56,-2.165)(7.56,-2.965)(7.96,-3.765)(7.96,-3.765)
\psline[linecolor=black, linewidth=0.04](9.56,-1.365)(9.96,-2.165)(9.96,-2.965)(9.56,-3.765)(9.56,-3.765)
\psline[linecolor=black, linewidth=0.04](7.96,-1.365)(9.56,-1.365)(9.56,-1.365)
\psline[linecolor=black, linewidth=0.04](7.96,-3.765)(9.56,-3.765)(9.56,-3.765)
\rput[bl](12.28,-1.145){1}
\rput[bl](13.06,-1.165){2}
\rput[bl](13.76,-2.265){3}
\rput[bl](13.1,-4.245){5}
\rput[bl](12.26,-4.245){6}
\rput[bl](13.76,-3.085){4}
\rput[bl](11.46,-4.225){7}
\rput[bl](10.8,-3.085){8}
\rput[bl](10.8,-2.305){9}
\rput[bl](11.32,-1.145){10}
\psdots[linecolor=black, dotsize=0.2](11.56,-1.365)
\psdots[linecolor=black, dotsize=0.2](12.36,-1.365)
\psdots[linecolor=black, dotsize=0.2](13.16,-1.365)
\psdots[linecolor=black, dotsize=0.2](13.56,-2.165)
\psdots[linecolor=black, dotsize=0.2](13.56,-2.965)
\psdots[linecolor=black, dotsize=0.2](13.16,-3.765)
\psdots[linecolor=black, dotsize=0.2](12.36,-3.765)
\psdots[linecolor=black, dotsize=0.2](11.56,-3.765)
\psdots[linecolor=black, dotsize=0.2](11.16,-2.965)
\psdots[linecolor=black, dotsize=0.2](11.16,-2.165)
\psline[linecolor=black, linewidth=0.04](11.56,-1.365)(11.16,-2.165)(11.16,-2.965)(11.56,-3.765)(11.56,-3.765)
\psline[linecolor=black, linewidth=0.04](13.16,-1.365)(13.56,-2.165)(13.56,-2.965)(13.16,-3.765)(13.16,-3.765)
\psline[linecolor=black, linewidth=0.04](11.56,-1.365)(13.16,-1.365)(13.16,-1.365)
\psline[linecolor=black, linewidth=0.04](11.56,-3.765)(13.16,-3.765)(13.16,-3.765)
\rput[bl](15.88,-1.145){1}
\rput[bl](16.66,-1.165){2}
\rput[bl](17.36,-2.265){3}
\rput[bl](16.7,-4.245){5}
\rput[bl](15.86,-4.245){6}
\rput[bl](17.36,-3.085){4}
\rput[bl](15.06,-4.225){7}
\rput[bl](14.4,-3.085){8}
\rput[bl](14.4,-2.305){9}
\rput[bl](14.92,-1.145){10}
\psdots[linecolor=black, dotsize=0.2](15.16,-1.365)
\psdots[linecolor=black, dotsize=0.2](15.96,-1.365)
\psdots[linecolor=black, dotsize=0.2](16.76,-1.365)
\psdots[linecolor=black, dotsize=0.2](17.16,-2.165)
\psdots[linecolor=black, dotsize=0.2](17.16,-2.965)
\psdots[linecolor=black, dotsize=0.2](16.76,-3.765)
\psdots[linecolor=black, dotsize=0.2](15.96,-3.765)
\psdots[linecolor=black, dotsize=0.2](15.16,-3.765)
\psdots[linecolor=black, dotsize=0.2](14.76,-2.965)
\psdots[linecolor=black, dotsize=0.2](14.76,-2.165)
\rput[bl](19.48,-1.145){1}
\rput[bl](20.26,-1.165){2}
\rput[bl](20.96,-2.265){3}
\rput[bl](20.3,-4.245){5}
\rput[bl](19.46,-4.245){6}
\rput[bl](20.96,-3.085){4}
\rput[bl](18.66,-4.225){7}
\rput[bl](18.0,-3.085){8}
\rput[bl](18.0,-2.305){9}
\rput[bl](18.52,-1.145){10}
\psdots[linecolor=black, dotsize=0.2](18.76,-1.365)
\psdots[linecolor=black, dotsize=0.2](19.56,-1.365)
\psdots[linecolor=black, dotsize=0.2](20.36,-1.365)
\psdots[linecolor=black, dotsize=0.2](20.76,-2.165)
\psdots[linecolor=black, dotsize=0.2](20.76,-2.965)
\psdots[linecolor=black, dotsize=0.2](20.36,-3.765)
\psdots[linecolor=black, dotsize=0.2](19.56,-3.765)
\psdots[linecolor=black, dotsize=0.2](18.76,-3.765)
\psdots[linecolor=black, dotsize=0.2](18.36,-2.965)
\psdots[linecolor=black, dotsize=0.2](18.36,-2.165)
\psline[linecolor=black, linewidth=0.04](18.76,-1.365)(18.36,-2.165)(18.36,-2.965)(18.76,-3.765)(18.76,-3.765)
\psline[linecolor=black, linewidth=0.04](20.36,-1.365)(20.76,-2.165)(20.76,-2.965)(20.36,-3.765)(20.36,-3.765)
\psline[linecolor=black, linewidth=0.04](18.76,-3.765)(20.36,-3.765)(20.36,-3.765)
\rput[bl](4.68,-6.345){1}
\rput[bl](5.46,-6.365){2}
\rput[bl](6.16,-7.465){3}
\rput[bl](5.5,-9.445){5}
\rput[bl](4.66,-9.445){6}
\rput[bl](6.16,-8.285){4}
\rput[bl](3.86,-9.425){7}
\rput[bl](3.2,-8.285){8}
\rput[bl](3.2,-7.505){9}
\rput[bl](3.72,-6.345){10}
\psdots[linecolor=black, dotsize=0.2](3.96,-6.565)
\psdots[linecolor=black, dotsize=0.2](4.76,-6.565)
\psdots[linecolor=black, dotsize=0.2](5.56,-6.565)
\psdots[linecolor=black, dotsize=0.2](5.96,-7.365)
\psdots[linecolor=black, dotsize=0.2](5.96,-8.165)
\psdots[linecolor=black, dotsize=0.2](5.56,-8.965)
\psdots[linecolor=black, dotsize=0.2](4.76,-8.965)
\psdots[linecolor=black, dotsize=0.2](3.96,-8.965)
\psdots[linecolor=black, dotsize=0.2](3.56,-8.165)
\psdots[linecolor=black, dotsize=0.2](3.56,-7.365)
\psline[linecolor=black, linewidth=0.04](3.96,-6.565)(3.56,-7.365)(3.56,-8.165)(3.96,-8.965)(3.96,-8.965)
\psline[linecolor=black, linewidth=0.04](5.56,-6.565)(5.96,-7.365)(5.96,-8.165)(5.56,-8.965)(5.56,-8.965)
\psline[linecolor=black, linewidth=0.04](3.96,-6.565)(5.56,-6.565)(5.56,-6.565)
\psline[linecolor=black, linewidth=0.04](3.96,-8.965)(5.56,-8.965)(5.56,-8.965)
\rput[bl](9.88,-6.345){1}
\rput[bl](10.66,-7.165){2}
\rput[bl](10.62,-8.625){3}
\rput[bl](8.8,-8.265){5}
\rput[bl](8.8,-7.505){6}
\psrotate(9.82, -9.425){-0.37204528}{\rput[bl](9.82,-9.425){4}}
\rput[bl](11.04,-7.165){7}
\rput[bl](11.84,-7.145){8}
\rput[bl](11.86,-8.625){9}
\rput[bl](11.02,-8.585){10}
\rput[bl](1.4,0.255){\large{$G_7$}}
\rput[bl](5.0,0.255){\large{$G_8$}}
\rput[bl](8.6,0.255){\large{$G_9$}}
\rput[bl](12.2,0.255){\large{$G_{10}$}}
\rput[bl](15.8,0.255){\large{$G_{11}$}}
\rput[bl](19.4,0.255){\large{$G_{12}$}}
\rput[bl](1.4,-4.945){\large{$G_{13}$}}
\rput[bl](5.0,-4.945){\large{$G_{14}$}}
\rput[bl](8.6,-4.945){\large{$G_{15}$}}
\rput[bl](12.2,-4.945){\large{$G_{16}$}}
\rput[bl](15.8,-4.945){\large{$G_{17}$}}
\rput[bl](19.4,-4.945){\large{$G_{18}$}}
\rput[bl](4.6,-10.145){\large{$G_{19}$}}
\rput[bl](10.2,-10.145){\large{$G_{20}$}}
\rput[bl](15.8,-10.145){\large{$G_{21}$}}
\psline[linecolor=black, linewidth=0.04](1.56,9.035)(1.56,6.635)(1.56,6.635)
\psline[linecolor=black, linewidth=0.04](2.36,9.035)(2.76,7.435)(2.76,7.435)
\psline[linecolor=black, linewidth=0.04](2.76,8.235)(2.36,6.635)(2.36,6.635)
\psline[linecolor=black, linewidth=0.04](0.76,9.035)(0.36,7.435)(0.36,7.435)
\psline[linecolor=black, linewidth=0.04](0.36,8.235)(0.76,6.635)(0.76,6.635)
\psline[linecolor=black, linewidth=0.04](5.16,9.035)(5.16,6.635)
\psline[linecolor=black, linewidth=0.04](5.96,9.035)(3.96,8.235)(3.96,8.235)
\psline[linecolor=black, linewidth=0.04](4.36,9.035)(3.96,7.435)(3.96,7.435)
\psline[linecolor=black, linewidth=0.04](6.36,8.235)(5.96,6.635)(5.96,6.635)
\psline[linecolor=black, linewidth=0.04](4.36,6.635)(6.36,7.435)(6.36,7.435)
\psbezier[linecolor=black, linewidth=0.04](7.96,9.035)(7.96,8.235)(9.56,8.235)(9.56,9.035)
\psline[linecolor=black, linewidth=0.04](7.56,8.235)(9.96,8.235)(9.96,8.235)
\psline[linecolor=black, linewidth=0.04](8.76,9.035)(9.96,7.435)(9.96,7.435)
\psline[linecolor=black, linewidth=0.04](8.76,6.635)(7.56,7.435)(7.56,7.435)
\psbezier[linecolor=black, linewidth=0.04](11.56,9.035)(11.56,8.235)(13.16,8.235)(13.16,9.035)
\psline[linecolor=black, linewidth=0.04](12.36,9.035)(13.56,8.235)(13.56,8.235)
\psline[linecolor=black, linewidth=0.04](13.56,7.435)(11.56,6.635)(11.56,6.635)
\psline[linecolor=black, linewidth=0.04](13.16,6.635)(11.16,7.435)(11.16,7.435)
\psline[linecolor=black, linewidth=0.04](12.36,6.635)(11.16,8.235)(11.16,8.235)
\psline[linecolor=black, linewidth=0.04](15.96,9.035)(17.16,8.235)(17.16,8.235)
\psbezier[linecolor=black, linewidth=0.04](15.16,9.035)(15.16,8.235)(16.76,8.235)(16.76,9.035)
\psline[linecolor=black, linewidth=0.04](17.16,7.435)(15.96,6.635)(15.96,6.635)
\psline[linecolor=black, linewidth=0.04](16.76,6.635)(14.76,7.435)(14.76,7.435)
\psline[linecolor=black, linewidth=0.04](15.16,6.635)(14.76,8.235)(14.76,8.235)
\psline[linecolor=black, linewidth=0.04](19.56,9.035)(19.56,6.635)(19.56,6.635)
\psline[linecolor=black, linewidth=0.04](20.36,9.035)(18.76,6.635)(18.76,6.635)
\psline[linecolor=black, linewidth=0.04](20.76,8.235)(18.36,7.435)(18.36,7.435)
\psline[linecolor=black, linewidth=0.04](20.76,7.435)(18.36,8.235)(18.36,8.235)
\psline[linecolor=black, linewidth=0.04](18.76,9.035)(20.36,6.635)(20.36,6.635)
\psline[linecolor=black, linewidth=0.04](1.56,3.835)(1.56,1.435)(1.56,1.435)
\psline[linecolor=black, linewidth=0.04](0.76,1.435)(2.36,3.835)(2.36,3.835)
\psline[linecolor=black, linewidth=0.04](2.36,1.435)(0.76,3.835)(0.76,3.835)
\psline[linecolor=black, linewidth=0.04](0.36,3.035)(2.76,3.035)(2.76,3.035)
\psline[linecolor=black, linewidth=0.04](0.36,2.235)(2.76,2.235)(2.76,2.235)
\psline[linecolor=black, linewidth=0.04](5.16,3.835)(5.16,1.435)(5.16,1.435)
\psline[linecolor=black, linewidth=0.04](5.96,3.835)(3.96,2.235)(3.96,2.235)
\psline[linecolor=black, linewidth=0.04](6.36,3.035)(4.36,1.435)(4.36,1.435)
\psline[linecolor=black, linewidth=0.04](6.36,2.235)(4.36,3.835)(4.36,3.835)
\psline[linecolor=black, linewidth=0.04](5.96,1.435)(3.96,3.035)(3.96,3.035)
\psline[linecolor=black, linewidth=0.04](8.76,3.835)(8.76,1.435)(8.76,1.435)
\psbezier[linecolor=black, linewidth=0.04](7.96,3.835)(7.96,3.035)(9.56,3.035)(9.56,3.835)
\psline[linecolor=black, linewidth=0.04](7.96,1.435)(9.96,3.035)(9.96,3.035)
\psline[linecolor=black, linewidth=0.04](9.96,2.235)(7.56,2.235)(7.56,2.235)
\psline[linecolor=black, linewidth=0.04](9.56,1.435)(7.56,3.035)(7.56,3.035)
\psline[linecolor=black, linewidth=0.04](12.36,3.835)(13.16,3.835)(13.16,3.835)
\psline[linecolor=black, linewidth=0.04](12.36,3.835)(12.36,1.435)(12.36,1.435)
\psline[linecolor=black, linewidth=0.04](12.36,3.835)(13.16,1.435)(13.16,1.435)
\psline[linecolor=black, linewidth=0.04](13.16,3.835)(13.56,3.035)(13.56,3.035)
\psline[linecolor=black, linewidth=0.04](13.16,3.835)(11.56,1.435)(11.56,1.435)
\psline[linecolor=black, linewidth=0.04](13.56,3.035)(13.56,2.235)(13.56,2.235)
\psline[linecolor=black, linewidth=0.04](13.56,3.035)(11.16,2.235)(11.16,2.235)
\psline[linecolor=black, linewidth=0.04](13.56,2.235)(13.16,1.435)(13.16,1.435)
\psline[linecolor=black, linewidth=0.04](13.56,2.235)(11.16,3.035)(11.16,3.035)
\psline[linecolor=black, linewidth=0.04](13.16,1.435)(11.56,3.835)(11.56,3.835)
\psline[linecolor=black, linewidth=0.04](12.36,1.435)(11.56,1.435)(11.56,1.435)
\psline[linecolor=black, linewidth=0.04](12.36,1.435)(11.56,3.835)(11.56,3.835)
\psline[linecolor=black, linewidth=0.04](11.56,1.435)(11.16,2.235)(11.16,2.235)
\psline[linecolor=black, linewidth=0.04](11.16,2.235)(11.16,3.035)(11.16,3.035)
\psline[linecolor=black, linewidth=0.04](11.16,3.035)(11.56,3.835)(11.56,3.835)
\psline[linecolor=black, linewidth=0.04](15.96,3.835)(15.96,1.435)(15.96,1.435)
\psbezier[linecolor=black, linewidth=0.04](15.16,3.835)(15.16,3.035)(16.76,3.035)(16.76,3.835)
\psline[linecolor=black, linewidth=0.04](17.16,3.035)(14.76,3.035)(14.76,3.035)
\psline[linecolor=black, linewidth=0.04](17.16,2.235)(14.76,2.235)(14.76,2.235)
\psbezier[linecolor=black, linewidth=0.04](7.96,6.635)(7.96,7.435)(9.56,7.435)(9.56,6.635)
\psbezier[linecolor=black, linewidth=0.04](15.16,1.435)(15.16,2.235)(16.76,2.235)(16.76,1.435)
\psline[linecolor=black, linewidth=0.04](19.56,3.835)(19.56,1.435)(19.56,1.435)
\psline[linecolor=black, linewidth=0.04](20.36,3.835)(20.76,2.235)(20.76,2.235)
\psline[linecolor=black, linewidth=0.04](20.76,3.035)(18.76,1.435)(18.76,1.435)
\psline[linecolor=black, linewidth=0.04](20.36,1.435)(18.36,3.035)(18.36,3.035)
\psline[linecolor=black, linewidth=0.04](18.76,3.835)(18.36,2.235)(18.36,2.235)
\psline[linecolor=black, linewidth=0.04](1.56,-1.365)(1.56,-3.765)(1.56,-3.765)
\psline[linecolor=black, linewidth=0.04](2.36,-3.765)(0.36,-2.965)(0.36,-2.965)
\psline[linecolor=black, linewidth=0.04](0.76,-3.765)(2.76,-2.965)(2.76,-2.965)
\psline[linecolor=black, linewidth=0.04](0.36,-2.165)(2.76,-2.165)(2.36,-2.165)
\psbezier[linecolor=black, linewidth=0.04](0.76,-1.365)(0.76,-2.165)(2.36,-2.165)(2.36,-1.365)
\psline[linecolor=black, linewidth=0.04](5.16,-1.365)(5.16,-3.765)(5.16,-3.765)
\psline[linecolor=black, linewidth=0.04](6.36,-2.165)(5.96,-3.765)(5.96,-3.765)
\psline[linecolor=black, linewidth=0.04](3.96,-2.165)(4.36,-3.765)(4.36,-3.765)
\psline[linecolor=black, linewidth=0.04](3.96,-2.965)(6.36,-2.965)(6.36,-2.965)
\psbezier[linecolor=black, linewidth=0.04](4.36,-1.365)(4.36,-2.165)(5.96,-2.165)(5.96,-1.365)
\psline[linecolor=black, linewidth=0.04](8.76,-1.365)(8.76,-3.765)(8.76,-3.765)
\psline[linecolor=black, linewidth=0.04](9.56,-1.365)(9.96,-2.965)(9.96,-2.965)
\psline[linecolor=black, linewidth=0.04](9.96,-2.165)(7.56,-2.165)(7.56,-2.165)
\psline[linecolor=black, linewidth=0.04](7.96,-1.365)(7.96,-3.765)(7.96,-3.765)
\psline[linecolor=black, linewidth=0.04](7.56,-2.965)(9.56,-3.765)(9.56,-3.765)
\psline[linecolor=black, linewidth=0.04](12.36,-1.365)(12.36,-3.765)(12.36,-3.765)
\psline[linecolor=black, linewidth=0.04](13.16,-1.365)(11.16,-2.165)(11.16,-2.165)
\psline[linecolor=black, linewidth=0.04](11.56,-1.365)(13.56,-2.165)(13.56,-2.165)
\psline[linecolor=black, linewidth=0.04](13.56,-2.965)(11.56,-3.765)(11.56,-3.765)
\psline[linecolor=black, linewidth=0.04](13.16,-3.765)(11.16,-2.965)(11.16,-2.965)
\psline[linecolor=black, linewidth=0.04](15.96,-1.365)(16.76,-1.365)(16.76,-1.365)
\psline[linecolor=black, linewidth=0.04](15.96,-1.365)(16.76,-3.765)(16.76,-3.765)
\psline[linecolor=black, linewidth=0.04](15.96,-1.365)(15.96,-3.765)(15.96,-3.765)
\psline[linecolor=black, linewidth=0.04](16.76,-1.365)(17.16,-2.165)(17.16,-2.165)
\psline[linecolor=black, linewidth=0.04](16.76,-1.365)(15.16,-3.765)(15.16,-3.765)
\psline[linecolor=black, linewidth=0.04](17.16,-2.165)(17.16,-2.965)(17.16,-2.965)
\psline[linecolor=black, linewidth=0.04](17.16,-2.165)(14.76,-2.965)(14.76,-2.965)
\psline[linecolor=black, linewidth=0.04](17.16,-2.965)(16.76,-3.765)(16.76,-3.765)
\psline[linecolor=black, linewidth=0.04](17.16,-2.965)(14.76,-2.165)(14.76,-2.165)
\psline[linecolor=black, linewidth=0.04](16.76,-3.765)(15.16,-1.365)(15.16,-1.365)
\psline[linecolor=black, linewidth=0.04](15.96,-3.765)(14.76,-2.965)(14.76,-2.965)
\psline[linecolor=black, linewidth=0.04](15.96,-3.765)(14.76,-2.165)(14.76,-2.165)
\psline[linecolor=black, linewidth=0.04](15.16,-3.765)(14.76,-2.165)(14.76,-2.165)
\psline[linecolor=black, linewidth=0.04](15.16,-3.765)(15.16,-1.365)(15.16,-1.365)
\psline[linecolor=black, linewidth=0.04](14.76,-2.965)(15.16,-1.365)(15.16,-1.365)
\psline[linecolor=black, linewidth=0.04](19.56,-1.365)(20.36,-1.365)(20.36,-1.365)
\psline[linecolor=black, linewidth=0.04](19.56,-1.365)(20.76,-2.165)(20.76,-2.165)
\psline[linecolor=black, linewidth=0.04](19.56,-1.365)(20.36,-3.765)(20.36,-3.765)
\psline[linecolor=black, linewidth=0.04](20.36,-1.365)(20.76,-2.965)(20.76,-2.965)
\psline[linecolor=black, linewidth=0.04](19.56,-3.765)(18.76,-1.365)(18.76,-1.365)
\psline[linecolor=black, linewidth=0.04](18.76,-1.365)(18.36,-2.965)(18.36,-2.965)
\psline[linecolor=black, linewidth=0.04](18.36,-2.165)(18.76,-3.765)(18.76,-3.765)
\psline[linecolor=black, linewidth=0.04](4.76,-6.565)(4.76,-8.965)(4.76,-8.965)
\psline[linecolor=black, linewidth=0.04](5.96,-7.365)(3.56,-8.165)(3.56,-8.165)
\psline[linecolor=black, linewidth=0.04](5.96,-8.165)(3.56,-7.365)(3.56,-7.365)
\psbezier[linecolor=black, linewidth=0.04](3.96,-6.565)(3.96,-7.365)(5.56,-7.365)(5.56,-6.565)
\psbezier[linecolor=black, linewidth=0.04](3.96,-8.965)(3.96,-8.165)(5.56,-8.165)(5.56,-8.965)
\psline[linecolor=black, linewidth=0.04](9.96,-8.965)(9.16,-8.165)(9.16,-7.365)(9.96,-6.565)(10.76,-7.365)(10.76,-8.165)(9.96,-8.965)(9.96,-8.965)
\psline[linecolor=black, linewidth=0.04](11.16,-8.165)(11.16,-7.365)(11.96,-7.365)(11.96,-8.165)(11.16,-8.165)(11.16,-8.165)
\psline[linecolor=black, linewidth=0.04](11.16,-7.365)(11.96,-8.165)(11.96,-8.165)
\psline[linecolor=black, linewidth=0.04](11.96,-7.365)(11.16,-8.165)(11.16,-8.165)
\psdots[linecolor=black, dotsize=0.2](9.96,-6.565)
\psdots[linecolor=black, dotsize=0.2](10.76,-7.365)
\psdots[linecolor=black, dotsize=0.2](10.76,-8.165)
\psdots[linecolor=black, dotsize=0.2](9.96,-8.965)
\psdots[linecolor=black, dotsize=0.2](9.16,-8.165)
\psdots[linecolor=black, dotsize=0.2](9.16,-7.365)
\psdots[linecolor=black, dotsize=0.2](11.16,-7.365)
\psdots[linecolor=black, dotsize=0.2](11.96,-7.365)
\psdots[linecolor=black, dotsize=0.2](11.96,-8.165)
\psdots[linecolor=black, dotsize=0.2](11.16,-8.165)
\rput[bl](15.48,-6.345){1}
\rput[bl](16.26,-7.165){2}
\rput[bl](16.22,-8.625){3}
\rput[bl](14.4,-8.265){5}
\rput[bl](14.4,-7.505){6}
\psrotate(15.42, -9.425){-0.37204528}{\rput[bl](15.42,-9.425){4}}
\rput[bl](16.64,-7.165){7}
\rput[bl](17.44,-7.145){8}
\rput[bl](17.46,-8.625){9}
\rput[bl](16.62,-8.585){10}
\psline[linecolor=black, linewidth=0.04](15.56,-8.965)(14.76,-8.165)(14.76,-7.365)(15.56,-6.565)(16.36,-7.365)(16.36,-8.165)(15.56,-8.965)(15.56,-8.965)
\psline[linecolor=black, linewidth=0.04](16.76,-8.165)(16.76,-7.365)(17.56,-7.365)(17.56,-8.165)(16.76,-8.165)(16.76,-8.165)
\psline[linecolor=black, linewidth=0.04](16.76,-7.365)(17.56,-8.165)(17.56,-8.165)
\psline[linecolor=black, linewidth=0.04](17.56,-7.365)(16.76,-8.165)(16.76,-8.165)
\psdots[linecolor=black, dotsize=0.2](15.56,-6.565)
\psdots[linecolor=black, dotsize=0.2](16.36,-7.365)
\psdots[linecolor=black, dotsize=0.2](16.36,-8.165)
\psdots[linecolor=black, dotsize=0.2](15.56,-8.965)
\psdots[linecolor=black, dotsize=0.2](14.76,-8.165)
\psdots[linecolor=black, dotsize=0.2](14.76,-7.365)
\psdots[linecolor=black, dotsize=0.2](16.76,-7.365)
\psdots[linecolor=black, dotsize=0.2](17.56,-7.365)
\psdots[linecolor=black, dotsize=0.2](17.56,-8.165)
\psdots[linecolor=black, dotsize=0.2](16.76,-8.165)
\psline[linecolor=black, linewidth=0.04](9.96,-6.565)(9.96,-8.965)(9.96,-8.965)
\psline[linecolor=black, linewidth=0.04](9.16,-7.365)(10.76,-7.365)(10.76,-7.365)
\psline[linecolor=black, linewidth=0.04](9.16,-8.165)(10.76,-8.165)(10.76,-8.165)
\psline[linecolor=black, linewidth=0.04](15.56,-6.565)(15.56,-8.965)(15.56,-8.965)
\psline[linecolor=black, linewidth=0.04](16.36,-8.165)(14.76,-7.365)(14.76,-7.365)
\psline[linecolor=black, linewidth=0.04](14.76,-8.165)(16.36,-7.365)(16.36,-7.365)
\end{pspicture}
}
\end{center}
\caption{Cubic graphs of order $10$.}\label{cubic}
	\end{figure}

As an immediate result of Lemma \ref{union}, we have the  following results:

\begin{proposition}
\begin{enumerate}
\item[(i)] If $e=v_rv_{r+1}\in E(P_n)$, then $ {E_{ESO}}(P_n-e)= {E_{ESO}}(P_r)+ {E_{ESO}}(P_s)$, where $r+s=n$.

\item[(ii)] If $e\in E(C_n)$, ($n\geq 3$), then
${E_{ESO}}(C_n-e)={E_{ESO}}(P_n).$

\item[(iii)] Let $S_n$ be the star on $n$ vertices and $e\in E(S_n)$. Then for any $n\geq 3$,
$${E_{ESO}}(S_n-e)={E_{ESO}}(S_{n-1}).$$
\end{enumerate}
\end{proposition}

Now consider the $2$-regulars. Every $2$-regular graph is a disjoint union of cycles. By Theorem \ref{thm-cycle}, we can find all the eigenvalues of elliptic Sombor matrix of cycle graphs. Therefore by Lemma \ref{union}, we can find elliptic Sombor characteristic polynomial and elliptic Sombor energy of 2-regular graphs. Before we continue, we need the following easy result that is the direct conclusion of the definition of elliptic Sombor energy:

\begin{proposition}\label{Prop:k-regular}
If  $G$ is a $k$-regular graph of order $n$, then
$${E_{ESO}}(G)=(2k)^n{E_{SO}}(G).$$
\end{proposition}

In \cite{Somborenergy}, we showed that the Sombor energy of $K_n$ is
${E_{SO}}(K_n)=2(n-1)^2\sqrt{2}$. So by Proposition \ref{Prop:k-regular}, we have the following result:

\begin{theorem}
For $n\geq 2$,
The elliptic Sombor energy of $K_n$ is
$${E_{ESO}}(K_n)=2^{n+1}(n-1)^{n+2}\sqrt{2}.$$
\end{theorem}

Now, we consider to  the elliptic characteristic polynomial of $3$-regular graphs of order $10$. Also we compute the elliptic {Sombor} energy  of this class of graphs. There are exactly $21$ cubic graphs of  order $10$ given in
Figure \ref{cubic} (see \cite{reza}).  We have the Sombor energy of these graphs in table 1 as we computed them in \cite{Somborenergy}:

\begin{center}
\begin{footnotesize}
\small
\begin{tabular}{|c|c|||c|c|||c|c|} \hline
$G_i$ & ${E_{SO}}(G_i)$ &$G_i$ & ${E_{SO}}(G_i)$ &$G_i$ & ${E_{SO}}(G_i)$ \\
\hline
\hline $G_1$ & 64.161 &  $G_8$ & 64.161  & $G_{15}$ & 62.767  \\
\hline $G_2$ & 63.043 &  $G_9$ & 64.981   & $G_{16}$ & 59.396  \\
\hline $G_3$ & 62.880 &  $G_{10}$  & 61.399  & $G_{17}$ & 67.882   \\
\hline $G_4$ & 57.336 &  $G_{11}$ & 62.375  & $G_{18}$ &  57.517 \\
\hline $G_5$ & 60.638 &  $G_{12}$ & 67.882 & $G_{19}$  & 66.096  \\
\hline $G_6$ & 63.403 &  $G_{13}$  & 61.000 & $G_{20}$ &  59.396 \\
\hline $G_7$ & 63.969 &  $G_{14}$  & 65.835 & $G_{21}$  & 50.911  \\
\hline
\end{tabular}
\end{footnotesize}
\end{center}
\begin{center}
{Table 1.} Sombor energy of cubic graphs of order $10$.
\end{center}

Now, by Using Table 1 and Proposition \ref{Prop:k-regular}, we have the elliptic Sombor energy of cubic graphs of order $10$ up to three decimal places, as we see in Table 2:

\begin{center}
\begin{footnotesize}
\small
\begin{tabular}{|c|c|||c|c|||c|c|} \hline
$G_i$ & ${E_{ESO}}(G_i)$ &$G_i$ & ${E_{ESO}}(G_i)$ &$G_i$ & ${E_{ESO}}(G_i)$ \\
\hline
\hline $G_1$ & 13858.776 &  $G_8$ & 13858.776  & $G_{15}$ & 13557.672  \\
\hline $G_2$ & 13617.288 &  $G_9$ & 14035.896   & $G_{16}$ & 12829.536  \\
\hline $G_3$ & 13582.080 &  $G_{10}$  & 13262.184  & $G_{17}$ & 14662.512   \\
\hline $G_4$ & 12384.576 &  $G_{11}$ & 13473.000  & $G_{18}$ &  12423.672 \\
\hline $G_5$ & 13097.808 &  $G_{12}$ & 14662.512 & $G_{19}$  & 12980.736  \\
\hline $G_6$ & 13695.048 &  $G_{13}$  & 13176.000 & $G_{20}$ &  12829.536 \\
\hline $G_7$ & 13758.336 &  $G_{14}$  & 14220.360 & $G_{21}$  & 10996.776  \\
\hline
\end{tabular}
\end{footnotesize}
\end{center}
\begin{center}
{Table 2.} Elliptic Sombor energy of cubic graphs of order $10$.
\end{center}

\begin{proposition}\label{prop-unique}
Six cubic graphs of order $10$   are not  ${\cal {E_{ESO}}}$-unique.
\end{proposition}

\begin{proof}
By Table 2, we see that $[G_1]=\{G_1,G_8\}$, $[G_{12}]=\{G_{12},G_{17}\}$ and $[G_{16}]=\{G_{16},G_{20}\}$. Therefore, we have fifteen cubic graphs of order $10$  which are ${\cal {E_{SO}}}$-unique.\quad\qed
\end{proof}

As an immediate result of Proposition \ref{prop-unique}, we have:

\begin{corollary}
In general, two $k$-regular graphs of the same order may have different  elliptic Sombor energy.
\end{corollary}

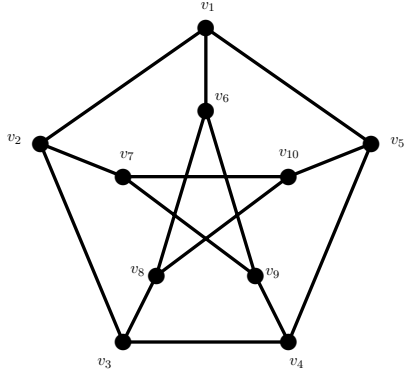
\begin{figure}[!h]
	\begin{center}
		\psscalebox{0.55 0.55}
{
\begin{pspicture}(0,-5.785)(9.63,3.085)
\psdots[linecolor=black, dotsize=0.4](4.8,2.435)
\psdots[linecolor=black, dotsize=0.4](4.8,0.435)
\psdots[linecolor=black, dotsize=0.4](2.8,-1.165)
\psdots[linecolor=black, dotsize=0.4](3.6,-3.565)
\psdots[linecolor=black, dotsize=0.4](6.0,-3.565)
\psdots[linecolor=black, dotsize=0.4](6.8,-1.165)
\psline[linecolor=black, linewidth=0.08](4.8,0.435)(3.6,-3.565)(6.8,-1.165)(2.8,-1.165)(6.0,-3.565)(4.8,0.435)(4.8,0.435)
\psdots[linecolor=black, dotsize=0.4](8.8,-0.365)
\psdots[linecolor=black, dotsize=0.4](0.8,-0.365)
\psdots[linecolor=black, dotsize=0.4](6.8,-5.165)
\psdots[linecolor=black, dotsize=0.4](2.8,-5.165)
\psline[linecolor=black, linewidth=0.08](4.8,2.435)(8.8,-0.365)(6.8,-5.165)(2.8,-5.165)(0.8,-0.365)(4.8,2.435)(4.8,0.435)(4.8,0.435)
\psline[linecolor=black, linewidth=0.08](6.8,-1.165)(8.8,-0.365)(8.8,-0.365)
\psline[linecolor=black, linewidth=0.08](6.0,-3.565)(6.8,-5.165)(6.8,-5.165)
\psline[linecolor=black, linewidth=0.08](3.6,-3.565)(2.8,-5.165)(2.8,-5.165)
\psline[linecolor=black, linewidth=0.08](2.8,-1.165)(0.8,-0.365)(0.8,-0.365)
\rput[bl](4.68,2.835){$v_1$}
\rput[bl](0.0,-0.365){$v_2$}
\rput[bl](2.18,-5.785){$v_3$}
\rput[bl](6.82,-5.785){$v_4$}
\rput[bl](9.26,-0.445){$v_5$}
\rput[bl](5.02,0.635){$v_6$}
\rput[bl](2.72,-0.765){$v_7$}
\rput[bl](2.98,-3.565){$v_8$}
\rput[bl](6.24,-3.605){$v_9$}
\rput[bl](6.56,-0.725){$v_{10}$}
\end{pspicture}
}
	\end{center}
	\caption{Petersen graph } \label{petersen}
\end{figure}

\begin{theorem}\label{Pet1}
Let ${\cal G}$ be the family of $3$-regular graphs of order $10$. For the Petersen graph $P$ (Figure \ref{petersen} or $G_{17}$ in Figure \ref{cubic}), we have the following properties:
\begin{itemize}
\item[(i)]
The Petersen graph $P$  is not ${\cal {E_{ESO}}}$-unique in ${\cal G}$.
\item[(ii)]
The Petersen graph $P$ has the maximum elliptic {Sombor} energy in ${\cal G}$.
\end{itemize}
\end{theorem}

\begin{proof}
\begin{itemize}
\item[(i)]
The Sombor matrix of $P$ is
$$A_{ESO}(P)=\left( \begin{array}{cccccccccc}
0&18\sqrt{2} &0 &0 &18\sqrt{2} &18\sqrt{2} &0 & 0&0 & 0 \\
18\sqrt{2}&0 &18\sqrt{2} &0 &0 &0 &18\sqrt{2} & 0&0 & 0 \\
0&18\sqrt{2} &0 &18\sqrt{2} &0 &0 &0 & 18\sqrt{2}&0 & 0 \\
0&0 &18\sqrt{2} &0 &18\sqrt{2} &0 &0 & 0&18\sqrt{2} & 0 \\
18\sqrt{2}&0 &0 &18\sqrt{2} &0 &0 &0 & 0&0 & 18\sqrt{2} \\
18\sqrt{2}&0 &0 &0 &0 &0 &0 & 18\sqrt{2}&18\sqrt{2} & 0 \\
0&18\sqrt{2} &0 &0 &0 &0 &0 & 0&18\sqrt{2} & 18\sqrt{2} \\
0&0 &18\sqrt{2} &0 &0 &18\sqrt{2} &0 & 0&0 & 18\sqrt{2} \\
0&0 &0 &18\sqrt{2} &0 &18\sqrt{2} &18\sqrt{2} & 0&0 & 0 \\
0&0 &0 &0 &18\sqrt{2} &0 &18\sqrt{2} & 18\sqrt{2}&0 & 0 \\
\end{array} \right).$$
So
\begin{align*}
\phi _{SO}(P,\lambda)&=det(\lambda I -A_{SO}(P))=(\lambda -9\sqrt{2})(\lambda +6\sqrt{2})^4 (\lambda-3\sqrt{2})^5.
\end{align*}
Therefore we have:
$$\lambda _1=1994\sqrt{2}~~,~~\lambda _2=\lambda _3=\lambda _4=\lambda _5=-1296\sqrt{2}~~,~~\lambda _6=\lambda _7=\lambda _8=\lambda _9=\lambda _{10}=648\sqrt{2},$$
and so we have ${E_{SO}}(P)=10368\sqrt{2}$. By Table 2, we have $P\in \{G_{12},G_{17}\}$. Therefore  $P$  is not ${\cal {E_{ESO}}}$-unique  in ${\cal G}$.
\item[(ii)]
It follows from Part (i) and Table 2.\quad\qed
\end{itemize}
\end{proof}

In \cite{Somborenergy}, we have shown that if two connected $k$-regular graphs have the same Sombor energy, then their adjacency matrices may have or have not the same permanent. Now by Proposition \ref{Prop:k-regular}, we have the following result:

{
\begin{proposition}\label{permanent}
If two connected $k$-regular graphs have the same elliptic Sombor energy, then their adjacency matrices may have or have not the same permanent.
\end{proposition}
}
	
Also in \cite{Somborenergy}, we have shown  that if two graphs have the same permanent, then we can not conclude that they have same Sombor energy. Consequently, if we have two graphs with the same permanent, then we can not conclude that they have same elliptic Sombor energy.

We think that the elliptic Sombor energy of no graph is integer. We end this section with the following conjecture:

\begin{conjecture}\label{conj3}
There is no graph with integer-valued elliptic Sombor energy.
\end{conjecture}

\section{Conclusions}

In this paper we  introduced the elliptic Sombor matrix and the elliptic Sombor energy of a graph $G$. We computed the elliptic Sombor characteristic polynomial and the elliptic Sombor energy for  some  graph classes. Also, we studied the elliptic Sombor energy of cubic graphs of order $10$.

\end{document}